\documentclass[11pt, a4paper]{amsart}
\usepackage{amsfonts}

\newtheorem{theorem}{\bf Theorem}[section]
\newtheorem{remark}{\bf Remark}[section]

\newtheorem{corollary}{Corollary}[section]
\newtheorem{spcase}{Special Case}[section]

\newtheorem{definition}{Definition}[section]

\usepackage{enumerate,mathrsfs}
\usepackage{amssymb,amsmath,graphicx,amsthm,mathtools,hyperref}

\usepackage{booktabs}

\theoremstyle{plain}
\usepackage{siunitx}

\usepackage{algorithm}
\usepackage{algorithmic,colortbl}

\makeindex
\usepackage[intoc]{nomencl} 

\makenomenclature

\usepackage{hyperref}
\hypersetup{nesting=true,debug=true,naturalnames=true}
\usepackage{graphicx,amssymb,upref}

\usepackage{subcaption}
\usepackage{enumerate,float}
\usepackage{fullpage}

\begin{document}

\title{ Fractional Generalizations of the Compound Poisson Process}
\author[]{Neha Gupta}
\address{\emph{Department of Theoretical Statistics and Mathematics Unit,
		Indian Statistical Institute Delhi, New Delhi - 110016, India.}}
\email{nehagupta@isid.ac.in}
\author[]{Aditya Maheshwari}
\address{\emph{Operations Management and Quantitative Techniques Area, Indian Institute of Management Indore, Indore 453556, India.}}
\email{adityam@iimidr.ac.in}


			\keywords{Time-fractional Poisson process, Compound Poisson process.}
			\subjclass{60G22, 60G51, 60G55. }
			
			\begin{abstract}
				This paper introduces the Generalized Fractional Compound Poisson Process (GFCPP), which claims to be a unified fractional version of the compound Poisson process (CPP) that encompasses existing variations as special cases. We derive its distributional properties, generalized fractional differential equations, and martingale properties. Some results related to the governing differential equation about the special cases of jump distributions, including exponential, Mittag-Leffler, Bernst\'ein, discrete uniform, truncated geometric, and discrete logarithm. Some of processes in the literature such as the fractional Poisson process of order $k$, P\'olya-Aeppli process of order $k$, and fractional negative binomial process becomes the special case of the GFCPP. Classification based on arrivals by time-changing the compound Poisson process by the inverse tempered and the inverse of inverse Gaussian subordinators are studied. Finally, we present the simulation of the sample paths of the above-mentioned processes.  
			\end{abstract}
			\maketitle
			
			\section{Introduction}
			
			The compound Poisson process (CPP) is one the most important stochastic models for count data with random jump events. It generalizes the classical unit jump size of the Poisson process to random jump size distribution. This model is distinguished by its ability to effectively represent real-world scenarios where the size or impact of events varies and it is particularly useful in situations where extreme events of random sizes play a crucial role. Therefore, it becomes a natural model for 
			wide range of applications in various sectors including insurance \cite{app:ins1,app:ins}, reliability \cite{app:reliability}, statistical physics \cite{app:statphy}, mining \cite{app:mining}, evolutionary biology \cite{app:evo-bio} and many more. Due to its wide applicability, it always remains a central of attraction of both theoretical and applied probabilists.  \\\\
			In this context, the generalizations of the CPP becomes an important problem to consider. Several attempts has been made to generalize the CPP by various authors and here we mention important results in the literature from fractional generalizations point of view. The first fractional generalization of the CPP was proposed by \cite{MeerSche2004,vietnam}, and a semi-Markov extension of the CPP was discussed by \cite{frac-CPP-Scalas}. Later, in \cite{beghinejp2009,beghin-2012}, the authors studied alternative forms of the compound fractional Poisson processes and derived several important results about their time-changed versions, governing differential equations and limiting behaviour of the processes. These fractional forms (see \cite{BegClau14}) are defined by time-changing the CPP by inverse stable subordinator and also by changing the jump size distribution. A multivariate extension of the fractional generalizations of the CPP is introduced in \cite{multi-cpp}. In \cite{CPP-poiss}, the Poisson subordinated CPP is considered. More recently, some fractional versions of the CPP by changing jump size and/or by  time-changing the CPP are examined in \cite{ayushicpp,fppok-1,kataria-cpp, kataria2022generalized}.\\\\
			After examining the existing literature on this topic, we felt a need for a unified fractional form of the CPP and therefore, we introduce a fractional generalization of the  CPP such that most of the studied CPP become a special case of the proposed process.  Our process is defined as   
			\begin{equation*}
				Y_{f}(t): = \sum_{i=1}^{N_{f}(t)} X_i,
			\end{equation*}
			where $N_{f}(t)\stackrel{d}{=}N(E_f(t))$ is the Poisson process time-changed with an independent inverse subordinator $\{E_{f}(t)\}_{t\geq 0}$ \eqref{inverse-sub} and $X_i$'s$,\;i=1,2,\ldots,$ are the iid jumps with common distribution $F_X$. We call this process as the  generalized fractional CPP (GFCPP). It is to note that we have assumed the most generalized form on the count and jump size distributions and therefore, this formulation will serve as an all-encompassing process for any type of  fractional generalization of the CPP. We compute the distributional properties and the governing generalized fractional differential equation of the probability mass function (\textit{pmf}) of the GFCPP. We also prove that the compensated GFCPP is a martingale with respect to its natural filtration. \\\\
			The special cases of the jump distribution $X_i,i=1,2,\ldots$, namely, exponential, Mittag-Leffler, Bernst\'ein, discrete uniform, truncated geometric, and discrete logarithm, are investigated. We obtain their Laplace Transform (LT), governing fractional differential equations and time-changed representations. Specifically to mention that this approach generalize the fractional Poisson process of order $k$ (\cite{tcppok,fppok-1}), the  P\'olya-Aeppli process of order $k$ (see \cite{chukova2015polya}), and  the fractional negative binomial process (see \cite{fnbpfp,BegClau14}). The classification of the GFCPP based on arrivals, particularly,  tempered fractional Poisson process and inverse of inverse Gaussian time-change of the  Poisson process are worked out. We further discuss their special cases based on jump sizes and obtain the LT and governing fractional differential equations. Lastly, we present the simulations for the special cases of the aforementioned processes. \\\\
			The paper is organized as follows. In Section \ref{sec:pre}, we present some preliminary definitions and results that are required for the rest of the paper. The GFCPP is examined in detailed in Section \ref{sec:gfcpp}. The classification of the special cases of the GFCPP based on jump sizes and arrivals are discussed in Sections \ref{sec:jump} and \ref{sec:arr}, respectively. In Section \ref{sec:sim}, the simulations of the sample paths are presented.

			\section{Preliminaries}\label{sec:pre}
			\noindent In this section, we introduce some notations and results that will be used later. \subsection{L\'evy subordinator and its inverse}
			
			A L\'evy subordinator (hereafter referred to as the subordinator) $\{D_{f}(t)\}_{t\geq0}$ is a non-decreasing L\'evy process and its Laplace transform (LT) (see \cite[Section 1.3.2]{appm}) has the form
			\begin{equation*}
				\mathbb{E}[e^{-s D_{f}(t)}]=e^{-tf(s)},
				\;{\rm where}\; 
				f(s)=b s+\int_{0}^{\infty}(1-e^{-s x})\nu(dx),~b\geq0, s>0,
			\end{equation*}
			is the Bernst\'ein function (see \cite{Bernstein-book} for more details). 
			Here $b$ is the drift coefficient and $\nu$ is a non-negative L\'evy measure on positive half-line satisfying 
			\begin{equation*}
				\int_{0}^{\infty}(x\wedge 1)\nu(dx)<\infty~~{\rm and}~~\nu([0,\infty))= \infty
			\end{equation*}
			which ensures that the sample paths of $\{D_{f}(t)\}_{t \geq 0}$ are almost surely $(a.s.)$  strictly increasing.
			Also, the first-exit time of $\{D_f(t)\}_{t\geq0}$ is defined as
			\begin{equation}\label{inverse-sub}
				E_{f}(t)=\inf\{r\geq 0:D_{f}(r)>t\},
			\end{equation}
			which is the right-continuous inverse of the subordinator $\{D_f(t)\}_{t\geq 0}$.
			The process $\{E_{f}(t)\}_{t \geq 0}$ is non-decreasing and its sample paths are continuous.
			\noindent We list some special cases of strictly increasing subordinators. The following subordinators with Laplace exponent denoted by $f(s)$ are very often used in literature.
			\begin{equation}\label{Levy_exponent}
				f(s) =
				\begin{cases}
					s^{\alpha},\; 0<\alpha<1, & $(stable subordinator)$;\\
					(s+\mu)^{\alpha}-\mu^{\alpha},\;\mu >0,\; 0<\alpha<1,&$(tempered stable  subordinator)$;\\
					\delta(\sqrt{2s+\gamma^2}-\gamma),\; \gamma>0,\; \delta>0, & $(inverse Gaussian subordinator)$.
				\end{cases}
			\end{equation}
			\subsection{Compound Poisson Process}
			A compound Poisson process is a continuous-time process with iid jumps $X_i,i=1,2,\ldots$. 
			A compound Poisson process $\{Y(t)\}_{t \geq 0}$ is given by
			\begin{equation}\label{CPP}    
				Y(t)=\sum_{i=1}^{N(t)}X_i,
			\end{equation}
			where $X_i$'s follows $F_X$ distribution and jumps arrive  randomly according to an independent Poisson process $N(t)$ with intensity rate $\lambda >0$. 
			\noindent The LT of the  CPP $\{Y(t)\}_{t\geq 0}$ is given by 
			\begin{align}\label{LT:CPP}
				\mathbb{E}[e^{-sY(t)}] = \mathbb{E}[e^{-\lambda t (1-\mathbb{E}[e^{-sX_1}])}].
			\end{align}
			The \textit{pmf} $P(n,t)=\mathbb{P}[Y(t)=n]$ of the  CPP is given by
			$$
			P(n,t)= \sum_{m=1}^{\infty}F_X^{*m}(n) \frac{e^{-\lambda t}(\lambda t)^m}{m!}, 
			$$
			where $F_X^{*m}$ is the $m$-fold convolution of the density of $F_X$.
			\subsection{Generalized fractional derivatives}
			Let $f$ be a Bernst\'ein function with integral representation 
			$$ f(s) =  \int_{0}^{\infty} (1-e^{-x s}) \nu(dx),\;\; s>0.$$
			We will use the generalized Caputo-Djrbashian (C-D) derivative with respect to the Bernst\'ein function $f$, which is defined on the space of absolutely continuous functions as follows (see \cite{Toa-gene}, Definition 2.4)
			\begin{equation}\label{Gen_Caputo_FD}
				\mathcal{D}^{f}_{t}u(t)=b\frac{d}{dt}u(t)+\int \frac{\partial}{\partial t}u(t-s)\Bar{\nu}(s)ds,
			\end{equation}
			where $\Bar{\nu}(s) = a + \Bar{\nu}(s, \infty)$ is the tail of the L\'evy measure.\\
			The generalized Riemann-Liouville (R-L)
			derivative according to the Bernst\'ein function $f$ as (see \cite[Definition 2.1]{Toa-gene})
			\begin{equation}\label{gen_R-L}
				\mathbb{D}^{f}_{t}u(t)=b\frac{d}{dt}u(t)+\frac{d}{dt}\int u(t-s)\Bar{\nu}(s)ds.
			\end{equation}
			The relation between $\mathbb{D}^{f}_{t}$ and $\mathcal{D}^{f}_{t}$ are given by (see \cite{Toa-gene}, Proposition 2.7)
			\begin{equation}\label{RElation_RL_C}
				\mathbb{D}^{f}_{t}u(t)= \mathcal{D}^{f}_{t}u(t)+\Bar{\nu}(t) u(0).
			\end{equation}
			\subsection{LRD for non-stationary process}\label{LRD_definition}
			Let $s>0$ be fixed and $t>s$. Suppose a stochastic process $\{X(t)\}_{t\geq0}$ has the correlation function Corr$(X(s),X(t))$ that satisfies
			\begin{equation}
				c_1(s)t^{-d}\leq\text{Corr}(X(s),X(t))\leq c_2(s)t^{-d},
			\end{equation}
			for large $t$, $d>0$, $c_1(s)>0$ and $c_2(s)>0$. In other words, 
			\begin{equation}
				\lim\limits_{t\to\infty}\frac{\text{Corr}(X(s),X(t))}{t^{-d}}=c(s),
			\end{equation}for some $c(s)>0$ and $d>0.$  We say $\{X(t)\}_{t\geq0}$ has the LRD property if $d\in(0,1)$.
			

		\section{Generalized Fractional Compound Poisson Processes}\label{sec:gfcpp}
		\noindent In this section, we introduce the generalized fractional compound Poisson process and study their properties. 
		
		\begin{definition}[Generalized fractional compound Poisson process ]\label{def:gcfpp}
			Let $N_{f}(t)\stackrel{d}{=}N(E_f(t))$ be the Poisson process time-changed with  the inverse subordinator $\{E_{f}(t)\}_{t \geq 0}$ (see \cite{mnv}) and let
			$X_i,\;i=1,2,\ldots,$ be the iid jumps with common
			distribution $F_X$. The process defined by
			\begin{equation}\label{GFCPP}
				Y_{f}(t) := \sum_{i=1}^{N_{f}(t)} X_i,t\geq0,
			\end{equation}
			is called the generalized fractional compound Poisson process (GFCPP).
		\end{definition}
		
		\noindent Using \eqref{LT:CPP}, we obtain the LT of the  \textit{pmf} of the GFCPP  $\{Y_{f}(t)\}_{t\geq0}$. It is given by
		\begin{align}\label{LT:GCFPP}
			\mathbb{E}[e^{-sY_{f}(t)}] &= \mathbb{E}\left[\mathbb{E}\left[\sum_{i=1}^{N_f(t)} X_i|E_{f}(t)\right]\right]= \mathbb{E}[e^{-\lambda E_{f}(t) (1-\mathbb{E}[e^{-sX_1}])}].
		\end{align}
		\noindent It is clear from the above LT that we can express the GFCPP as a time-changed CPP $$ Y_{f}(t)\stackrel{d}{=}Y(E_f(t)),t\geq0,$$ where  $Y(t)=\sum_{i=0}^{N(t)} X_i$ denotes the CPP. \\
		\noindent Next, we present the governing generalized fractional differential equation of the \textit{pmf} of the GFCPP.
		\begin{theorem}\label{Pro_DDE_GFCPP}
			The \textit{pmf} $P_{f}(n,t)=\mathbb{P}[Y_f(t)=n]$  satisfies following fractional differential equation
			\begin{align}\label{Diff_Gcpp}
				\mathcal{D}^{f}_{t} P_{f}(n,t) &= -\lambda  P_{f}(n,t) +\lambda \int_{-\infty}^{\infty} P_{f}(n-x,t) F_X(x) dx.
			\end{align}
		\end{theorem}
		\begin{proof} Let $h_{f}(y,t)$ be the probability density function (\textit{pdf}) of inverse subordinator $\{E_f(t)\}_{t \geq 0}$ and $P(n,t)$ be \textit{pmf} of CPP. Using conditional argument, we have that
			$$
			\mathbb{P}(Y_{f}(t) \in dz)=\int_{0}^{\infty} \mathbb{P}(Y(x) \in dz) h_{f}(x,t) dy.
			$$
			We now take generalized Riemann-Liouville
			derivative given by \eqref{gen_R-L} on both sides of above equation, we get
			\begin{align}
				\mathbb{D}^{f}_{t} P_{f}(n,t)&=  \int_{0}^{\infty}p(n,x)\mathbb{D}^{f}_{t}h_f(x,t) dx.\nonumber\\
				&=-\int_{0}^{\infty}P(n,y)\frac{\partial}{\partial x}h_f(x,t)dx \nonumber\\
				&= -P(n,x)h_f(x,t)|_{0}^{\infty} + \int_{0}^{\infty} \frac{\partial}{\partial x}P(n,x)h_f(x,t) dx. \label{rl-pmf1}
			\end{align}
			It is known that (see \cite{FPP-IS}) the distribution of CPP satisfies the following partial differential equation 
			\begin{equation}\label{diff-cpp}
				\frac{\partial}{\partial y}P(n,x) = \lambda P(n,x)+ \lambda \int_{-\infty}^{+\infty}P(n-u,y)F_X(u)du.
			\end{equation}
			\noindent Substituting \eqref{diff-cpp} in \eqref{rl-pmf1} and subsequently using the relation \eqref{RElation_RL_C}, we obtain the result mentioned in \eqref{Diff_Gcpp}. This completes the proof.
		\end{proof}
		\noindent Further, we discuss some distributional properties of the GFCPP.
		\begin{theorem}\label{m_v_c_GFCPP}
			
			The mean, variance and covariance of GFCPP is given by:
			\begin{enumerate}[(i)]
				\item $ \mathbb{E}[Y_f(t)] =\lambda \mathbb{E}[E_{f}(t)]\mathbb{E}[X_1]$
				\item ${\rm Var}[Y_f(t)] = \lambda \mathbb{E}[E_f(t)]\mathbb{E}[X_1^2]+\lambda^2(\mathbb{E}[X_1])^2{\rm Var}[E_f(t)]$
				\item ${\rm Cov}[Y_f(t),Y_f(s)]  = \lambda \mathbb{E}[X_1^2]\mathbb{E}[E_f(s)]+\lambda^2(\mathbb{E}[X_1])^2{\rm Cov}[E_f(t),E_f(s)].$
			\end{enumerate}
			\begin{proof} Using the conditional argument and independence of $N_f(t)$ and $X_i$, we have that 
				$$
				\mathbb{E}[Y_f(t)]= \mathbb{E}[N_f(t)]\mathbb{E}[X_1]=\lambda \mathbb{E}[E_f(t)]\mathbb{E}[X_1].
				$$
				The variance of $\{Y_f(t)\}_{t \geq 0}$ can be written as (see   \cite{LRD2014})
				$$
				{\rm Var}[Y_f(t)] = {\rm Var}[X_1]\mathbb{E}[N_f(t)]+\mathbb{E}[X_1]{\rm Var}[N_f(t)].
				$$
				Next, we compute the ${\rm Cov}[Y_f(t),Y_f(s)]$,  $s\leq t$,
				\begin{align*}
					{\rm Cov}[Y_f(t),Y_f(s)]&= \mathbb{E}\left[\sum_{k=1}^{\infty}X_k^2 \mathbb{I}\{N_f(s)\geq k\}\right]+\mathbb{E}\left[ \sum_{i\neq j}\sum X_i X_k \mathbb{I}\{N_f(s)\geq k, N_f(t)\geq i\}\right]\\
					& \hspace{3cm}-(\mathbb{E}[X_1])^2\mathbb{E}[N_f(s)]\mathbb{E}[N_f(t)]\\
					&= \mathbb{E}[X_1^2]\sum_{k=1}^{\infty}\mathbb{P}(N_f(s)\geq k)\\
					&\hspace{3cm}+(\mathbb{E}[X_1])^2\left[\sum_{i=1}^{\infty}\sum_{k=1}^{\infty}\mathbb{P}(N_f(s)\geq k, N_f(t)\geq i)-\mathbb{P}(N_f(s)\geq k)\right]\\
					&\hspace{3cm}-(\mathbb{E}[X_1])^2\mathbb{E}[N_f(s)]\mathbb{E}[N_f(t)]\\
					&=  \mathbb{E}[X_1^2] \mathbb{E}[N_f(s)]+(\mathbb{E}[X_1])^2\mathbb{E}[N_f(s) N_f(t)]\\
					&\hspace{3cm}-(\mathbb{E}[X_1])^2\mathbb{E}[N_f(s)]-(\mathbb{E}[X_1])^2\mathbb{E}[N_f(s)]\mathbb{E}[N_f(t)]\\
					&= {\rm Var}[X_1]\mathbb{E}[N_f(s)]+(\mathbb{E}[X_1])^2{\rm Cov}[N_f(t),N_f(s)].
				\end{align*}
				The expression for the covariance of time-changed Poisson process, that is ${\rm Cov}[N_f(t),N_f(s)]$, is derived in \cite{LRD2014}. Substituting the same in the above equation, we get the desired result.
			\end{proof}
		\end{theorem}
		\noindent Next, we prove the martingale property for the compensated GCFPP. Consider the compensated GCFPP defined by 
		\begin{equation*}
			M_{f}(t):=Y_{f}(t)-\lambda E_f(t) \mathbb{E}[X_1],\;\; t\geq 0.
		\end{equation*}
		
		\begin{theorem}
			Let $\mathbb{E}[X_i]<\infty, i=1,2,\ldots$.  The compensated GCFPP $\{M_{f}(t)\}_{t\geq 0} $ is a martingale with respect to a natural filtration $\mathscr{F}_{t}= (N_f(s), s\leq t) \vee \sigma(E_{f}(s), s\leq t)$. 
			\begin{proof}
				It is to note that (see \cite{fppok-1}) the compensated time-changed Poisson process $Q(t):=N_f(t)-\lambda E_{f}(t)$ is a martingale with respect to the filtration $\mathscr{F}_{t}  = (N_f(s), s\leq t) \vee \sigma(E_{f}(s), s\leq t)$. We have that 
				\begin{align*}
					\mathbb{E}[M_{f}(t)-M_{f}(s)|\mathscr{F}_{s}] &=\mathbb{E}\left[\sum_{i=1}^{N_{f}(t)} X_i - \lambda \mathbb{E}[X_1]E_f(t) -\left(\sum_{i=1}^{N_{f}(s)} X_i-\lambda \mathbb{E}[X_1] E_f(s) \right) \Bigg| \mathscr{F}_{s} \right]\\
					&= \mathbb{E}\left[\left(\sum_{i=N_{f}(s)+1}^{N_{f}(t)} X_i - \lambda a (E_f(t)-E_f(s))\right)\Bigg|\mathscr{F}_{s}\right]\\
					& = \mathbb{E}\left[\mathbb{E}\left[\sum_{i=N_{f}(s)+1}^{N_{f}(t)} X_i\Bigg| \mathscr{F}_{t}\right] \Bigg| \mathscr{F}_{s} \right] - \lambda a \mathbb{E}\left[E_f(t)-E_f(s)\Bigg| \mathscr{F}_{s}\right]\\
					& = \mathbb{E}\left[a (N_f(t)-N_f(s))-  \lambda a (E_f(t)-E_f(s))| \mathscr{F}_{s} \right]\\
					&=a\mathbb{E}\left[Q(t)-Q(s)| \mathscr{F}_{s}\right]\\
					&=0,
				\end{align*}
				since the $\{Q(t)\}_{t\geq0}$ is a $\mathscr{F}_{t}$-martingale. This completes the proof.
			\end{proof}
		\end{theorem}
		
		\section{Generalized time-fractional compound Poisson process}\label{sec:jump}
		\noindent In the previous section, we discussed some important properties and results related to the GFCPP. Note that Definition \ref{def:gcfpp} assumes general distribution $F_X$ on jump size. In this section, we consider several special cases of distribution on the jump size $X_i,i=1,2,\ldots$ of the GFCPP  and study their properties. We call this process as generalized time-fractional CPP (GTFCPP) and it can also be expressed as the CPP time-changed  with the inverse subordinator $\{Y(E_f(t))\}_{t \geq 0}$, where $\{Y(t)\}_{t \geq 0}$ and $\{E_f(t)\}_{t \geq 0}$ are independent with specific jump distribution in the CPP.
		
		\subsection{GTFCPP with exponential jumps} In this subsection, we assume that the jumps $X_i,\;i=1,2,\ldots,$ are exponentially distributed with parameter $\eta>0$ and denote it by 
		\begin{equation}\label{Time_CFPP}
			Y_f^\eta(t) := \sum_{i=1}^{N_{f}(t)} X_i, t\geq0.
		\end{equation}
		
		\noindent Using \eqref{LT:GCFPP}, we obtain the LT of the density $P^{\eta}_{f}(x,t)$, given by
		\begin{align}\label{LT_GTFCPP}
			\mathcal{L}[P^{\eta}_{f}(x,t)] = \mathbb{E}[e^{-sY_f^\eta(t)}] = \mathbb{E}[e^{-\lambda E_{f}(t) \frac{s}{s+\eta}}].
		\end{align}
		Next, we derive the differential equation associated with the density $P^{\eta}_{f}(x,t)$ of the GTFCPP with exponential jumps $\{Y_f^\eta(t)\}_{t\geq 0}$.
		\begin{theorem}
			The \textit{pdf} $P^{\eta}_{f}(x,t)$ of $\{Y_f^\eta(t)\}_{t\geq 0}$ satisfy following fractional differential equation
			\begin{align}\label{Diff_GFCPP_expo}
				\eta \mathcal{D}^{f}_{t} P^{\eta}_{f}(x,t) &= -\left[\lambda +\mathcal{D}^{f}_{t}\right]\frac{\partial}{\partial x}  P^{\eta}_{f}(x,t).
			\end{align}
			with following conditions
			$$
			P^{\eta}_{f}(x,0)=0, \; \; 
			\mathbb{P}(Y_f^\eta (t) >0)= 1-\mathbb{E}[e^{-\lambda E_f(t)}].
			$$
			\begin{proof}Consider the subordinated form of the \textit{pdf} $P^{\eta}_{f}(x,t)$
				$$
				P^{\eta}_{f}(x,t)=\int_{0}^{\infty} P(x,y)h_f(y,t) dy,
				$$
				where $h_{f}(y,t)$ is \textit{pdf} of inverse subordinator $\{E_f(t)\}_{t \geq 0}$.
				Taking generalized Riemann-Liouville
				derivative given by \eqref{gen_R-L}, we get
				\begin{align}
					\mathbb{D}^{f}_{t} P^{\eta}_{f}(x,t)&=  \int_{0}^{\infty} P_{Y}(x,y)\mathbb{D}^{f}_{t}h_f(y,t) dy.\nonumber\\
					&=-\int_{0}^{\infty} P_{Y}(x,y)\frac{\partial}{\partial y}h_f(y,t)dy\nonumber\\
					&= - P_{Y}(x,y)h_f(y,t)|_{0}^{\infty} + \int_{0}^{\infty} \frac{\partial}{\partial y} P_{Y}(x,y)h_f(y,t) dy.\label{eq:thm41}
				\end{align}
			Note that (see \cite{beghin-2012}) the \textit{pdf} $P_{Y}(x,t)$ of CPP satisfies the following equation  
				$$
				\eta \frac{\partial}{\partial t}P_{Y}(x,t)=-\left[\lambda +  \frac{\partial}{\partial t} \right]\frac{\partial}{\partial x}P_{Y}(x,t).
				$$
				Substituting the above equation in \eqref{eq:thm41}  and using \eqref{RElation_RL_C}, we obtain the desired result.
			\end{proof}
		\end{theorem}
		\noindent As a special case of Theorem \ref{m_v_c_GFCPP}, the mean and covariance of the GTFCPP with exponential jumps $\{Y_f^\eta(t)\}_{t\geq 0}$ can be found as follows
		\begin{align*}
			\mathbb{E}[Y_f^\eta(t)] &=\frac{\lambda}{\eta} \mathbb{E}[E_{f}(t)];\\
			{\rm Var}[Y_f^{\eta}f(t)] &= \frac{2\lambda}{\eta^2} \mathbb{E}[E_f(t)]+\frac{\lambda^2}{\eta^2}{\rm Var}[E_f(t)];\\
			{\rm Cov}[Y_f^{\eta}(t),Y_f^{\eta}(s)] &= \frac{\lambda}{\eta^2} \mathbb {E}[E_f(s)]+\frac{1}{\eta^2}{\rm Cov}[N_f(t),N_f(s)],\;\; s<t.
		\end{align*}
		\subsection{ GTFCPP with Mittag-Leffler jumps} We now define the process
		\begin{equation*}
			Y_f^{\beta,\eta}(t) := \sum_{i=1}^{N_{f}(t)} X_i, t\geq0,
		\end{equation*}
		where jump size $X_i,\;i=1,2,\ldots,$ are the Mittag-Leffler distributed random variables with parameter $\beta$ and $\eta$ and having the \textit{pdf} as $q_{\beta,\eta}(x,t)= \eta x^{\beta-1}E_{\beta,\beta}(-\eta x^{\beta}), \; \; \beta \in(0,1], \; \eta >0$. \\
		The LT of the $\{Y_f^{\eta} (t)\}_{t\geq 0}$ is given by
		\begin{equation}\label{LT:CPP_ML}
			\mathcal{L}[P^{\beta, \eta}_f(x,t)] =  \mathbb{E}[e^{-\lambda E_{f}(t) \frac{s^\beta}{s^\beta+\eta}}].
		\end{equation}
		We now derive a time-changed representation of the GTFCPP with Mittag-Leffler jumps $\{Y_f^{\beta,\eta}(t)\}_{ t\geq 0}$.
		\begin{theorem}
			Consider an $\beta$-stable subordinator $\{D_\beta(t)\}_{t\geq0}$ time-changed with   an independent GTFCPP with exponential jump $\{Y_f^{\eta} (t)\}_{t\geq 0}$, i.e.  
			$$ Y^{\eta,\beta}_{f}(t) \stackrel{d}{=} D_\beta(Y_f^{\eta} (t)),\; \beta \in (0,1),\; \eta > 0, t\geq0.$$
		\end{theorem}
		\begin{proof}
			The LT of the \textit{pdf} of $\{Y^{\eta,\beta}_{f}(t)\}_{t \geq 0}$ is
			$$
			\mathbb{E}[e^{-sD_\beta(Y_f^{\eta} (t))}] = \mathbb{E}[\mathbb{E}[e^{-sD_\beta(Y_f^{\eta} (t))}|Y_f^{\eta} (t)]]= \mathbb{E}[e^{ -s^{\beta}y_{f}^{\eta}(t)}].
			$$
			Using \eqref{LT_GTFCPP}, we get
			$$
			\mathbb{E}[e^{-sD_\beta(Y_f^{\eta} (t))}]=\mathbb{E}[e^{-\lambda E^{f}(t) \frac{s^\beta}{s^\beta+\eta}}].
			$$
			Comparing the above equation with the LT \eqref{LT:CPP_ML} of the $\{Y_f^{\beta,\eta}(t)\}_{t\geq 0}$, we get the desired result.
		\end{proof}
		\noindent Next, we present the fractional differential equation for the \textit{pmf} of the GTFCPP with Mittag-Leffler jumps $\{Y_f^{\beta,\eta}(t)\}_{t\geq 0}.$
		\begin{theorem} 
			The \textit{pdf} $P^{\beta,\eta}_f(x,t)$ of $\{Y_f^{\beta,\eta}(t)\}_{t\geq 0}$ satisfy following fractional differential equation
			\begin{align*}
				\eta \mathcal{D}^{f}_{t} P^{\beta,\eta}_f(x, t) &= -\left[\lambda +\mathcal{D}^{f}_{t}\right] \mathcal{D}^{\beta}_xP^{\beta,\eta}_f(x, t),
			\end{align*}
			where $\mathcal{D}^{\beta}_t$ denotes the C-D fractional derivative (a special case of \eqref{Gen_Caputo_FD}) with the following conditions
			$$
			P^{\beta,\eta}_f(x, 0)=0, \; \; 
			\mathbb{P}[Y_f^{\beta,\eta} (t) >0]= 1-\mathbb{E}[e^{-\lambda E_f(t)}].
			$$
			\begin{proof} 
				Writing the \textit{pdf} $P^{\beta,\eta}_f(x,t)$ using the conditional probability approach and then taking a generalized R-L fractional derivative,  we have that 
				\begin{align}
					\mathbb{D}^{f}_{t} P^{\beta,\eta}_f(x,t)&=  \int_{0}^{\infty} l_{\beta}(x,y)\mathbb{D}^{f}_{t}P_{f}^{\eta}(y,t) dy, \;\;(\text{using \eqref{Diff_GFCPP_expo}})\nonumber\\
					&=-\frac{1}{\eta}\left[\lambda +\mathcal{D}^{f}_{t}\right]\int_{0}^{\infty}l_\beta(x,y)\frac{\partial}{\partial y}P_{f}^{\eta}(y,t) dy\nonumber\\
					&= -\frac{1}{\eta}\left[\lambda +\mathcal{D}^{f}_{t}\right] \left(-l_{\beta}(x,y) P_{f}^{\eta}(y,t)|_{0}^{\infty} + \int_{0}^{\infty} \frac{\partial}{\partial y} l_\beta(x,y)P_{f}^{\eta}(y,t) dy\right),\label{eq:thm43}
				\end{align}
				where $l_{\beta}(x,t)$ is the \textit{pdf} of $\beta$-stable subordinator which satisfies the equation
				
				\begin{equation}\label{eq:thm43-2}
					\mathcal{D}^{\beta}_{x}l_{\beta}(x,t)=-\frac{\partial}{\partial t}l_{\beta}(x,t),\;\;l_{\beta}(x,0)=\delta(x).\end{equation}
				We obtain the desired result by substituting \eqref{eq:thm43-2} and   \eqref{RElation_RL_C} in \eqref{eq:thm43}.
			\end{proof}
		\end{theorem}
		\subsection{ GTFCPP with Bernst\'ein jumps}
		\noindent  In this subsection, we assume 
		that the jump size $X_i, \; i=1,2,\ldots$ of the GFCPP are  distributed as follows
		\begin{equation}\label{Dis_Bern}
			\mathbb{P}(X_i >t)= \mathbb{E}[e^{\eta E_g(t)}].
		\end{equation}
		Note  that the above distribution also occurs as the distribution of the inter-arrivals between the consecutive jumps of the time-changed Poisson process $\{N(E_g(t))\}_{t \geq 0}$, where $\{E_g(t)\}_{t \geq 0}$ is an independent inverse subordinator with Bernst\'ein function $g$ (see \cite{FPP-IS} for more details). Now, we define the process
		\begin{equation*}
			Y_f^{g}(t):= \sum_{i=1}^{N_{f}(t)} X_i, t\geq 0.
		\end{equation*}
		It is called as the  GTFCPP with Bernst\'ein jumps. 
		The LT of the $\{Y_{f}^{g}(t)\}_{t\geq 0}$ is given by
		\begin{align}\label{LT_TSFCPP-1}
			\mathbb{E}[e^{-sY_{f}^{g}(t)}]= \mathbb{E}\left[e^{-\lambda E^{f}(t) \frac{g(s)}{g(s)+\eta}}\right],
		\end{align}
		where
		$$\mathbb{E}[e^{-sX_1}]=  \frac{\eta}{g(s)+\eta}.
		$$
		Next, we obtain the time changed representation and the governing fractional differential equation of the GTFCPP with Bernst\'ein jumps $\{Y_{f}^{g}(t)\}_{t\geq 0}$ in the following Theorem. 
		\begin{theorem}\label{TC_Diff}
			Let $\{D_{g}(t)\}_{t \geq 0}$ be a L\'evy subordinator with Bernst\'ein function $g$ and $\{Y_{f}^{\eta}(t)\}_{t \geq 0}$ be the GTFCPP with exponential jumps \eqref{Time_CFPP}, independent of $\{D_{g}(t)\}_{t \geq 0}.$ Then 
			\begin{equation}\label{relation_tscfpp-1}
				Y_{f}^{g}(t)\stackrel{d}{=}D_{g}(Y_{f}^{\eta}(t))
			\end{equation}
			The \textit{pdf} $P_{f}^{g}(x,t)$ of $\{Y_{f}^{g}(t)\}_{t\geq 0}$ satisfies the following equation
			\begin{align*}
				\eta \mathcal{D}^{f}_{t} P_{f}^{g}(x,t) &= -\left[\lambda +\mathcal{D}^{f}_{t}\right]\mathcal{D}^{g}_{x} P_{f}^{g}(x,t),\; 
			\end{align*}
			with conditions
			$$
			P_{f}^{g}(x,0)=0, \; \; 
			\mathbb{P}(Y_f^g (t) >0)= 1-\mathbb{E}[e^{-\lambda E_f(t)}].
			$$
			\begin{proof} 
				The relation \eqref{relation_tscfpp-1} can be proved  by taking the LT of $\{D_{g}(Y_{f}^{\eta}(t))\}_{t \geq 0}$, which is given by
				$$\mathbb{E}[e^{-sD_{g}(Y_{f}^{\eta}(t))}]= \mathbb{E}[e^{-\lambda E^{f}(t) \frac{^{g(s)}}{g(s)+\eta}}]$$
				This is equal to the LT given in \eqref{LT_TSFCPP-1}. Hence by the uniqueness of LT, the result follows. Now, we express the \textit{pdf} $P_{f}^{g}(x,t)$ using the conditional probability approach  and  then take a generalized R-L fractional derivative on both sides. We have that 
				\begin{align}
					\mathbb{D}^{f}_{t} P_{f}^{g}(x,t)&=  \int_{0}^{\infty} l_{g}(x,y)\mathbb{D}^{f}_{t}P_{f}^{\eta}(y,t) dy, (\text{from \eqref{Diff_GFCPP_expo}})\nonumber\\
					&=-\frac{1}{\eta}\left[\lambda +\mathcal{D}^{f}_{t}\right]\int_{0}^{\infty}l_g(x,y)\frac{\partial}{\partial y}P_{f}^{\eta}(y,t) dy\nonumber\\
					&= -\frac{1}{\eta}\left[\lambda +\mathcal{D}^{f}_{t}\right] \left(-l_{g}(x,y) P_{f}^{\eta}(y,t)|_{0}^{\infty} + \int_{0}^{\infty} \frac{\partial}{\partial y} l_g(x,y)P_{f}^{\eta}(y,t) dy\right),\label{eq:thm44}
				\end{align} where $l_{g}(x,t)$ is the density of L\'evy subordinator which satisfies the equation (see \cite{Toa-gene}) $$\mathcal{D}^{f}_{x}l_{g}(x,t)=-\frac{\partial}{\partial t}l_{g}(x,t),\;\;l_{g}(x,0)=\delta(x).$$
				We substitute the above equation in \eqref{eq:thm44} and subsequently use \eqref{RElation_RL_C} to get the desired result.  
			\end{proof}
		\end{theorem}
		\subsection{Generalized fractional Poisson process of order $k$}
		In this subsection, we assume that jumps 
 $X_i,\;i=1,2,\ldots,$ are iid discrete uniform random variables such that $\mathbb{P}(X_i = j) = \frac{1}{k},\; j=1,2,\ldots, k$ and $\{N_f(t)\}_{t \geq 0}$ be time-changed Poisson process with intensity rate of the Poisson process $\{N(t)\}_{t\geq 0}$ is assumed to be  $k\lambda$. The process \eqref{GFCPP} can be written as 
		\begin{equation}\label{CFPPOK}
			Y_{f}^{k}(t): = \sum_{i=1}^{N_{f}(t)} X_i, t\geq 0,
		\end{equation}
		is called the generalized fractional Poisson process of order $k$ (GFCPPoK). This process was first defined and studied in \cite{fppok-1}. \\
		The time-changed representation of GFCPPoK is given as (see \cite{fppok-1})
		$$
		Y_f^{k}(t) \stackrel{d}{=} N^k(E_f(t)), t\geq 0,
		$$
		where $\{N^k(t)\}_{t \geq 0}$ is Poisson process of order $k$ (PPoK).\\
		The \textit{pmf} $P^{k}_{f}(n,t)=\mathbb{P}[Y_f^{k}(t)=n]$ of $\{Y_{f}^{k}(t)\}_{t \geq 0}$ satisfy following fractional differential-difference equation (see \cite[Proposition 7.4]{fppok-1}),
		\begin{align*}
			\mathcal{D}^{f}_{t} P^{k}_{f}(n,t) &= -k \lambda \left(1-\frac{1}{k}\sum_{j=1}^{n \wedge k} B^{j}\right)P^{k}_{f}(n,t),\; n>0,\\
			\mathcal{D}^{f}_{t}P^{k}_{f}(0,t)  &= - k \lambda P^{k}_{f}(0,t),\nonumber
		\end{align*}
		where $B$ is the backward shift operator i.e. $B[P(n,t)]=P(n-1,t)$.

		\subsection{Generalized P\'olya-Aeppli process of order $k$ }
		Consider the GTFCPP with $X_i$'s $i=1,2,\ldots$ as iid truncated geometrically distributed random variables with success probability $1-\rho$ and \textit{pmf} given by 
		$$
		\mathbb{P}[X_i=j] = \frac{1-\rho}{1-\rho^k}\rho^{j-1},\;\; j=1,2,\ldots,k, \;\; \rho \in [0,1). 
		$$
		The LT of $ X_1$ given by 
		$$
		\mathbb{E}[e^{-sX_1}] = \frac{(1-\rho)e^{-s}}{(1-\rho^k)} \frac{1-\rho^k e^{-s^k}}{1-\rho{e^{-s}}}.
		$$
		Note that when $k \rightarrow \infty $ the truncated geometric distribution approaches the geometric distribution starting at $1$ and success probability $1-\rho$.   
		We denote the process  as
		\begin{equation}\label{tsfcpp}
			Y_{f}^{\rho,k}(t) := \sum_{i=1}^{N_{f}(t)} X_i, t\geq 0.
		\end{equation}
		This is called the generalized Pólya-Aeppli process of order $k$ (GPAPoK).
		The LT of the $\{Y_{f}^{\rho,k}(t)\}_{t \geq 0}$ is given by 
		\begin{align*}
			\mathbb{E}[e^{-sY_{f}^{\rho}(t) }]= \mathbb{E}[e^{-\lambda E^{f}(t) (1-\mathbb{E}[e^{-sX_1}])}]=\mathbb{E}\left[e^{-\lambda E^{f}(t)\left(1- \frac{(1-\rho)e^{-s}}{(1-\rho^k)} \frac{1-\rho^k e^{-s^k}}{1-\rho{e^{-s}}}\right) }\right].
		\end{align*}
	The time-changed representation of GPAPoK is 
	$$
	Y_f^{\rho,k}(t) \stackrel{d}{=}N^k_A(E_f(t)), t\geq 0,
	$$
	where $\{N^k_A(t)\}_{t\geq0}$ is P\'olya-Aeppli process of order $k$ (PAPoK) (see \cite{chukova2015polya}). Further, we derive the differential equation for the GPAPoK.
	\begin{theorem}\label{Diff_GFPAOK}
		The \textit{pmf} $P^{\rho,k}_{f}(n,t)=\mathbb{P}[Y_f^{\rho,k}(t)=n]$ satisfy following fractional differential equation
		\begin{align*}
			\mathcal{D}^{f}_{t} P^{\rho,k}_{f}(n,t) &= -\lambda \left(1-\frac{1-\rho}{1-\rho^k}\sum_{j=1}^{n \wedge k}\rho^{j-1}B^j\right)  P^{\rho,k}_{f}(n,t), \; \; n=1,2, \ldots,\\
			\mathcal{D}^{f}_{t} P^{\rho,k}_{f}(0,t) &= -\lambda P^{\rho,k}_{f}(0,t),
		\end{align*}
		with an  initial condition
		$
		P^{\rho,k}_{f}(n,0)=\delta_{n,0}.
		$
		\begin{proof} Writing the  \textit{pmf} $P^{\rho,k}_{f}(n,t)$ using the condition probability approach, we have that  
			$$
			P^{\rho,k}_{f}(n,t)=\int_{0}^{\infty} P_{\rho,k}(n,y)h_f(y,t) dy,
			$$
			where the $P_{\rho,k}(n,t)$ is the \textit{pmf} of PAPoK $\{N^k_A(t)\}_{t\geq0}$.
			Taking generalized Riemann-Liouville 
			derivative \eqref{gen_R-L} on both sides of the above equation, we get
			\begin{align}
				\mathbb{D}^{f}_{t} P^{\rho, k}_{f}(n,t)&=  \int_{0}^{\infty} P_{\rho,k}(n,y)\mathbb{D}^{f}_{t}h_f(y,t) dy,\;\;(\text{see in \cite{Toa-gene}})\nonumber \\
				&=-\int_{0}^{\infty} P_{\rho,k}(n,y)\frac{\partial}{\partial y}h_f(y,t)dy\nonumber\\
				&= - P_{\rho,k}(n,y)h_f(y,t)|_{0}^{\infty} + \int_{0}^{\infty} \frac{\partial}{\partial y} P_{\rho,k}(n,y)h_f(y,t) dy.\label{eq:thrm45}
			\end{align}
			We know that (see \cite{Minkova}) the \textit{pmf } $P_{\rho,k}(n,t)$ of  the P\'olya-Aeppli process of order $k$ $\{N^k_A(t)\}_{t\geq 0}$ satisfies the following differential equation 
			$$
			\frac{\partial}{\partial t}P_{\rho,k}(n,t)=-\lambda \left(1-\frac{1-\rho}{1-\rho^k}\sum_{j=1}^{n \wedge k}\rho^{j-1}B^j \right)P_{\rho,k}(n,t),\;\; n\geq 1.
			$$
			Substituting the above equation in \eqref{eq:thrm45}  and using \eqref{RElation_RL_C}, we obtain the desired result.
		\end{proof}
	\end{theorem}
	\subsection{Fractional negative binomial process}
	Let $X_{i}$'s $i=1,2,\ldots$ be a sequence of iid random variables with discrete logarithmic distribution, given by
	$$
	\mathbb{P}[X_i=n]= \frac{-1}{\log (1-q)}\frac{q^n}{n},\;\; n\geq 1,\; q \in(0,1).
	$$
	We denote the process as
	\begin{equation*}
		Y^{q}_{f}(t): = \sum_{i=0}^{N_f(t)}X_i, t\geq 0.
	\end{equation*}
	It is also known (see \cite{BegClau14}) as the fractional negative binomial process with parameter $(1-q, \frac{\lambda}{\log{1-q}})$.
	The LT of the $\{Y^{q}_{f}(t)\}_{t \geq 0}$ is given by
	\begin{align*}
		\mathbb{E}[e^{-sY^{q}_{\lambda}(t) }]= \mathbb{E}[e^{-\lambda E^{f}(t) (1-(s+\mu)^\alpha -\mu^\alpha))}].
	\end{align*}

	\section{Classifications based on arrivals}\label{sec:arr}
	\noindent In this section, we work out special cases of the GFCPP, defined in \eqref{GFCPP}, by taking particular cases of time-changed Poisson process $\{N_f(t)\}_{t\geq 0}$. More specifically, we study two types of inverse subordinator $\{E_f(t)\}_{t \geq 0}$, namely the inverse tempered  $\alpha$-stable subordinator (ITSS) and the inverse of the inverse Gaussian (IG) subordinator. The distribution of jumps is assumed in a general sense. Further, some results are mentioned by taking special cases of the jump distributions of $X_i,i=1,2,\ldots$.
	\subsection{Tempered fractional CPP}\label{subsec:tmfcpp}
	\begin{definition}
		Consider the inverse subordinator \eqref{inverse-sub} associated with tempered stable Bernst\'ein function \eqref{Levy_exponent} $f(s)= (\mu+s)^\alpha-\mu^\alpha,\; \alpha \in (0,1], \mu>0$, denoted by $\{E_{\alpha,\mu}(t)\}_{t\geq0}$. Let $N_{\alpha,\mu}(t):=N(E_{\alpha,\mu}(t)),t\geq 0$ be the tempered fractional Poisson process (TFPP)  (studied in \cite{tsfpp}). The process \eqref{GFCPP} is defined by 
		\begin{align}\label{TTSFCPP}
			Y_{\alpha,\mu}(t):= \sum_{i=1}^{N_{\alpha,\mu}(t)}X_i, t\geq 0,
		\end{align}
		is called tempered fractional CPP (TFCPP) with  $X_i,\;i=1,2,\ldots,$ be the iid jumps having common
		distribution $F_X $.
	\end{definition}
	\noindent The \textit{pmf} $P_{\alpha,\mu}(n,t)=\mathbb{P}[Y_{\alpha,\mu}(t)=n]$ satisfy following tempered fractional differential equation (see \cite{FPP-IS})
	\begin{align*}
		\mathcal{D}^{\alpha,\mu}_{t} P_{\alpha, \mu}(n,t) &= -\lambda  P_{\alpha,\mu}(n,t) +\lambda \int_{-\infty}^{\infty} P_{\alpha,\mu}(n-x,t) F_X(x) dx,
	\end{align*}
	where $\mathcal{D}^{\alpha, \mu}_t$ denotes the tempered C-D fractional derivative (a special case of \eqref{Gen_Caputo_FD}).
	We next give some distributional results for the TFCPP.
	\begin{theorem}
		The mean, variance, and covariance of the TFCPP $\{Y_{\alpha,\mu}(t)\}_{t\geq 0}$ is given by
		\begin{enumerate}[(i)]
			\item   $ \mathbb{E}[Y_{\alpha, \mu}(t)]  =\lambda \mathbb{E}[X_1]\sum_{n=0}^{\infty}\frac{\mu^{\alpha}\gamma(\mu t; \alpha(1+n))}{\Gamma(\alpha(1+n))};$
			
			\item ${\rm Var}[Y_{\alpha, \mu}(t)] = \lambda \mathbb{E}[X_1^2]\sum_{n=0}^{\infty}\frac{\mu^{\alpha}\gamma(\mu t; \alpha(1+n))}{\Gamma(\alpha(1+n))}+\lambda^2(\mathbb{E}[X_1])^2{\rm Var}[E_f(t)]$
			
			\item ${\rm Cov}[Y_{\alpha, \mu}(t),Y_{\alpha, \mu}(s)]  = \lambda \mathbb{E}[X_1^2]\sum_{n=0}^{\infty}\frac{\mu^{\alpha}\gamma(\mu s; \alpha(1+n))}{\Gamma(\alpha(1+n))}+\lambda^2(\mathbb{E}[X_1])^2{\rm Cov}[E_f(t),E_f(s)].$
		\end{enumerate}
		where $\gamma(a;b)$ is an incomplete gamma function. 
		\begin{proof}The results follows from Theorem \eqref{m_v_c_GFCPP} by substituting the value of $\mathbb{E}[E_{\alpha,\mu}(t)]$ (see \cite{Mijena2014})
			$$
			\mathbb{E}[E_{\alpha, \mu}(t)]=\sum_{n=0}^{\infty}\frac{\mu^{\alpha}\gamma(\mu t; \alpha(1+n))}{\Gamma(\alpha(1+n))}.\qedhere
			$$
		\end{proof}
	\end{theorem}
	\begin{corollary}
		Let $\mathbb{E}[X_i]=0, i=1,2,\ldots$, then the correlation of the process is given by 
		$$
		{\rm Corr}[Y_{\alpha, \mu}(t), Y_{\alpha, \mu}(s)]=\sqrt{\frac{\mathbb{E}[E_{\alpha, \mu}(s)]}{\mathbb{E}[E_{\alpha, \mu}(t)]}},
		$$
		$$ \mathbb{E}[E_{\alpha, \mu}(t)] \sim \frac{t}{\alpha \mu^{\alpha-1}}, \text{as}\; {t\to\infty}, \text{(see in \cite{Mijena2014})}.
		$$
		Using \eqref{LRD_definition}, we obtain the correlation function of $Y_{\alpha, \mu}(t)$ and $Y_{\alpha, \mu}(s)$. It exhibits LRD property, i.e.
		$$
		\lim\limits_{t\to\infty}\frac{{\rm Corr}[Y_{\alpha, \mu}(t), Y_{\alpha, \mu}(s)]}{t^{-1/2}}  \sim  \mu^{(\alpha-1)/2}\alpha^{1/2} \sqrt{\mathbb{E}[E_{\alpha,\mu}(s)]}.
		$$
	\end{corollary}

	\noindent Next, we discuss special cases for the TFCPP, where $X_i$ follows some particular type of distribution.
	\begin{spcase}
		
		When $X_i,\;i=1,2,\ldots,$ follow exponential distribution with parameter $\eta >0$. The  process  $\{Y_f^\eta(t)\}_{t \geq 0}$ \eqref{TTSFCPP} can be represented in the following notation 
		\begin{align}\label{tem_cpp_ex}
			Y_{\alpha, \mu}^\eta(t):= \sum_{i=1}^{N_{\alpha, \mu}(t)} X_i ,t\geq 0.
		\end{align}
		This process $\{Y_{\alpha, \mu}^\eta(t)\}_{t \geq 0}$ can also be written in the time-changed representation as $\{Y(N_{\alpha, \mu}(t))\}_{t \geq 0}$.
		\noindent The \textit{pdf} $P_{\alpha, \mu}^{\eta}(x,t)$ satisfies the following equation
		\begin{align*}
			\eta \mathcal{D}^{\alpha, \mu}_{t} P_{\alpha, \mu}^{\eta}(x,t) &= -\left[\lambda +\mathcal{D}^{\alpha, \mu}(t)\right]\frac{\partial}{\partial x}  P_{\alpha, \mu}^{\eta}(x,t).
		\end{align*}
		where $ \mathcal{D}^{\alpha, \mu}_{t}$ is the tempered C-D derivative of order $\alpha \in (0,1)$ with tempering parameter $\mu>0$ is defined by
		$$
		\mathcal{D}^{\alpha, \mu}_{t}g(t) =  \frac{1}{\Gamma(1-\alpha)}\frac{d}{dt}\int_{0}^{t}\frac{g(u)du}{(t-u)^{\alpha}} - \frac{g(0)}{\Gamma(1-\alpha)}\int_{t}^{\infty}e^{-\mu r}\alpha r^{-\alpha-1}dr.
		$$
		with conditions
		$$
		P_{\alpha, \mu}^{\eta}(x,0)=0, \; \; 
		\mathbb{P}(Y_{\alpha, \mu}^\eta (t) >0)= 1-e^{-\lambda t}.
		$$
	\end{spcase}
	
	\begin{remark}
		When $\mu=0$, the tempered $\alpha$-stable subordinator reduces to the $\alpha$-stable subordinator. Then the process \eqref{tem_cpp_ex} becomes the time-fractional compound Poisson process, as defined in \cite{beghin-2012}. 
		\noindent The mean and covariance are given by 
		\begin{align*}
			\mathbb{E}[Y_{\alpha}^{\eta}(t)] &=\frac{\lambda t^\alpha}{\eta \Gamma(1+\alpha)};\\
			{\rm Cov}[Y_{\alpha}^{\eta}(t),Y_{\alpha}^{\eta}(s)] &= \frac{2\lambda s^{\alpha}}{\eta^2 \Gamma(1+\alpha)} +\frac{\lambda^2}{\eta^2}{\rm Cov}[E_\alpha(t),E_\alpha(s)],\;\; s<t.
		\end{align*}
	\end{remark}
	\begin{spcase} When
		$X_i,\; i=1,2,\ldots$ are iid random variables with tempered Mittag-Leffler distribution (\cite{tem_mit}) 
		$$
		f_{\beta, \eta, \nu}(x)= \lambda e^{-\nu x}\sum_{n=0}^{\infty} (-1)^n(\lambda-\nu^\beta)^n \frac{x^{\beta(n+1)-1}}{\Gamma(\beta(n+1))}, \; \lambda >\nu^{\beta}, x>0.
		$$
		Then, process $\eqref{GFCPP}$ is defined as
		\begin{align*}
			Y_{\alpha,\mu}^{\beta, \eta, \nu}(t): = \sum_{i=1}^{N_{\alpha,\mu}(t)}X_{i},t\geq 0, 
		\end{align*}
		where $N_{\alpha,\mu}(t) = N(E_{\alpha, \mu}(t))$ is tempered fractional Poisson process with rate parameter $\lambda>0$ (see \cite{tsfpp}). It is called a  TFCPP with tempered Mittag-Leffler jumps.
		The LT of \textit{pdf} of $\{Y_{\alpha,\mu}^{\beta, \eta, \nu}(t)\}_{t \geq 0}$ is given by 
		\begin{align*}
			\mathbb{E}[e^{-sY_{\alpha,\mu}^{\beta, \eta, \nu}(t) }]= \mathbb{E}[e^{-\lambda E_{\alpha, \mu}(t)\frac{(s+\nu)^\beta-\nu^\beta}{\eta+(s+\nu)^\beta-\nu^\beta}}],
		\end{align*}
		where $$
		\mathbb{E}[e^{-sX_1}]= \frac{\eta}{\eta+(s+\nu)^\beta-\nu^\beta}.
		$$
		An alternate representation of $\{Y_{\alpha,\mu}^{\beta, \eta, \nu}(t)\}_{ t \geq 0}$ is given by  time-changing the tempered $\beta$-stable subordinator $\{D_{\beta, \nu}(t)\}_{t \geq 0}, \; \beta \in (0,1]$ with $\{Y_{\alpha, \mu}^{\eta}(t)\}_{t \geq 0}$, i.e. 
		$$
		Y_{\alpha,\mu}^{\beta, \eta, \nu}(t) \stackrel{d}{=} D_{\beta, \nu}(Y_{\alpha, \mu}^{\eta}(t)).
		$$
		The \textit{pdf} $P_{\alpha,\mu}^{\beta, \eta, \nu}(x,t)$ satisfies the following fractional differential equation
		\begin{align*}
			\eta \mathcal{D}^{\alpha, \mu}_{t} P_{\alpha,\mu}^{\beta, \eta, \nu}(x,t) &= -\left[\lambda +\mathcal{D}^{\alpha, \mu}_{t}\right]\mathcal{D}^{\beta, \nu}_{x} P_{\alpha,\mu}^{\beta, \eta, \nu}(x,t)\; 
		\end{align*}
		with initial condition
		$$
		P_{\alpha,\mu}^{\beta, \eta, \nu}(x,0)=0, \; \; 
		\mathbb{P}(Y_{\alpha,\mu}^{\beta, \eta, \nu} (t) >0)= 1-\mathbb{E}[e^{-\lambda E_{\alpha, \mu}(t)}].
		$$
	\end{spcase}

	\begin{remark} Here, we discuss the particular values of parameter of the above-introduced processes.
		\begin{itemize}
			\item For $\nu=0$, then the process  is called time-changed  $\beta$-stable process, i.e. $\{D_\beta(Y_{\alpha, \mu}^{\eta}(t))\}_{t \geq 0}$.
			\item  For $\mu=0$, the process behaves as time-changed tempered $\beta$-stable subordinator, i.e. $\{D_{\beta, \nu}(Y_{\alpha}^{\eta}(t))\}_{t \geq 0}$ 
			\item  When $\mu=0,\; \nu=0$, then the process behaves as time-changed in $\beta$-stable subordinator $\{D_{\beta}(t)\}_{t \geq 0}$ with $\{Y_{\alpha}^{\eta}(t)\}_{t \geq 0}$, i.e. $\{D_{\beta}(Y_{\alpha}^{\eta}(t))\}_{t \geq 0}$ (see \cite{beghin-2012}).
		\end{itemize}
	\end{remark}
	\begin{spcase}
		Let $X_i, i=1,2,\ldots,$ are as distributed in \eqref{Dis_Bern}. Then the process 
		\begin{equation*}
			Y_{\alpha,\mu}^{g}(t) := \sum_{i=1}^{N_{\alpha,\mu}(t)} X_i,t\geq 0,
		\end{equation*}
		is called TFCPP with Bernst\'ein jumps.
		\noindent The subsequent result follow from Theorem \eqref{TC_Diff} as a particular case. This process has a time-changed representation, i.e.
		\begin{equation*}
			Y_{\alpha,\mu}^{g}(t) \stackrel{d}{=}D_{g}(Y_{\alpha,\mu}^{\eta}(t)).
		\end{equation*}
		The \textit{pdf} $P_{\alpha,\mu}^{g}(x,t)$ of $\{Y_{\alpha,\mu}^{g}(t)\}_{t \geq 0}$ satisfies the following equation
		\begin{align*}
			\eta \mathcal{D}^{\alpha,\mu}_{t} P_{\alpha,\mu}^{g}(x,t) &= -\left[\lambda +\mathcal{D}^{\alpha,\mu}_{t}\right]\mathcal{D}^{g}_{x} P_{\alpha,\mu}^{g}(x,t)\; 
		\end{align*}
		with conditions
		$$
		P_{\alpha,\mu}^{g}(x,0)=0, \; \; 
		\mathbb{P}(Y_{\alpha,\mu}^g (t) >0)= 1-\mathbb{E}[e^{-\lambda E_{\alpha,\mu}(t)}].
		$$
	\end{spcase}

	\begin{spcase}
		Let $X_i,i=1,2,\ldots $ be iid truncated geometrically distributed random variables. Using \eqref{tsfcpp}, we define the process
		$$
		Y_{\alpha,\mu}^{\rho,k}(t): = \sum_{i=1}^{N_{\alpha,\mu}(t)} X_i,t\geq 0.
		$$
		It  is called as the  tempered fractional PAPoK.
		Next, we mentioned the particular case of the  Theorem \eqref{Diff_GFPAOK} for the tempered fractional PAPoK. The \textit{pmf} $P^{\rho,k}_{\alpha,\mu}(n,t)=\mathbb{P}[Y_{\alpha,\mu}^{\rho,k}(t)=n]$ satisfy following fractional differential equation
		\begin{align*}
			\mathcal{D}^{\alpha,\mu}_{t} P^{\rho,k}_{\alpha,\mu}(n,t) &= -\lambda \left(1-\frac{1-\rho}{1-\rho^k}\sum_{j=1}^{n \wedge k}\rho^{j-1}B^j\right)  P^{\rho,k}_{\alpha,\mu}(n,t),\\
			\mathcal{D}^{\alpha,\mu}_{t} P^{\rho,k}_{\alpha,\mu}(0,t) &= -\lambda P^{\rho,k}_{\alpha,\mu}(0,t).
		\end{align*}
		\noindent The time-changed representation of $ \{Y_{\alpha,\mu}^{\rho,k}(t)\}_{t \geq 0}$ by time-changing in PAPoK $\{N^{k}_A(t)\}_{t \geq 0}$ with $\{E_{\alpha,\mu}(t)\}_{t \geq 0}$ such that
		$$
		Y_{\alpha,\mu}^{\rho,k}(t) \stackrel{d}{=}N^k_A(E_{\alpha,\mu}(t)) ,t\geq 0.
		$$
	\end{spcase}
	\begin{remark}
		When $\mu=0$, then $\{Y_{\alpha, 0}^{\rho,k}(t)\}_{t \geq 0}$ is called fractional FPAPok, defined in \cite{kadankova2023fractional}.
	\end{remark}
	\begin{spcase}
		Let $X_i,\;i=1,2,\ldots,$ be the discrete uniform distributed random variables. From equation \eqref{CFPPOK}, reduces the tempered fractional PPoK $\{Y_{\alpha, \mu}^k(t)\}_{t \geq 0}$, which is introduced and studied in \cite{fppok-1}.
	\end{spcase}
	\subsection{Inverse IG fractional CPP}\label{subsec:iig}
	\begin{definition}
		Consider the inverse subordinator \eqref{inverse-sub} associated with inverse Gaussian  Bernst\'ein function \eqref{Levy_exponent} $f(s)=\delta(\sqrt{2s+\gamma^2}-\gamma)$, denoted by $\{E_{\delta,\gamma}(t)\}_{t\geq 0}$. Let $N_{\delta,\gamma}(t):=N(E_{\delta,\gamma}(t)),t\geq 0$ be the inverse IG process (see \cite{Kumar-Hitting}). The process \eqref{GFCPP} defined by 
		\begin{align}\label{def:IG-FCPP}
			Y_{\delta,\gamma}(t):= \sum_{i=1}^{N_{\delta,\gamma}(t)}X_i,t\geq 0,
		\end{align}
		is called as the inverse IG fractional CPP with  $X_i,\;i=1,2,\ldots,$ be the iid jumps having common distribution $F_X $.
	\end{definition}
	\noindent Let $\{D_{\delta,\gamma}(t)\}_{t \geq 0}$ be IG L\'evy process with the LT (see \cite{ContTan2004})
	\begin{equation*}
		\mathbb{E}(e^{-s D_{\delta,\gamma}(t)}) = e^{-t \delta(\sqrt{2s+\gamma^2}-\gamma)}.
	\end{equation*}
	The L\'evy measure $\nu_{\delta,\gamma}$ corresponding to the inverse Gaussian  subordinator is given by (see \cite{ContTan2004})
	$$
	\nu_{\delta,\gamma}(dx)= \frac{\delta}{\sqrt{2\pi x^3}}e^{-\gamma^2x/2}\mathbb{I}_{\{x>0\}}dx.
	$$
	\noindent Further, we define convolution type fractional derivative or non-local operator corresponding to the inverse of IG subordinator, we have 
	\begin{align*}
		\Bar{\nu}_{\delta,\gamma}(s)  = \nu_{\delta,\gamma}(s,\infty) &= \int_{s}^{\infty}\frac{\delta}{\sqrt{2\pi x^3}}e^{-\gamma^2x/2}du, \; s>0,\nonumber\\
		&= \sqrt{\frac{2}{\pi}}\delta s^{-1/2}e^{-\gamma^2s/2}-\frac{\delta \gamma}{\sqrt{\pi}}\Gamma{\left(1/2;\gamma^2s/2\right)},
	\end{align*}
	where $\Gamma(a;b)= \int_{b}^{\infty} u^{a-1}e^{-u}dt$ is the upper incomplete gamma function.
	Using \eqref{Gen_Caputo_FD}, the generalized C-D fractional derivative corresponding to the IG subordinator is of the form 
	\begin{equation}\label{caputo_IG}
		\mathcal{D}^{\delta, \gamma}_t V(t)= \frac{d}{dt}\int_{0}^{t} \ v(s)\left( \sqrt{\frac{2}{\pi}}\delta (t-s)^{-1/2}e^{-\gamma^2s/2}-\frac{\delta \gamma}{\sqrt{\pi}}\Gamma{\left(1/2,\gamma^2(t-s)/2\right)}\right)ds- v(0)\Bar{\nu}_{\delta,\gamma}(t).
	\end{equation}
	The generalized R-L
	derivative corresponding to the inverse of IG subordinator is
	\begin{equation}\label{RL_IG}
		\mathbb{D}^{\delta, \gamma}_t v(t)= 
		\frac{d}{dt}\int_{0}^{t} \ v(s)\left( \sqrt{\frac{2}{\pi}}\delta (t-s)^{-1/2}e^{-\gamma^2(t-s)/2}-\frac{\delta \gamma}{\sqrt{\pi}}\Gamma{\left(1/2,\gamma^2(t-s)/2\right)}\right)ds.
	\end{equation}
	
	\begin{theorem}
		The \textit{pdf}  $h_{\delta,\gamma}(x,t)$ of $\{E_{\delta,\gamma}(t)\}_{t \geq 0}$ solves the following fractional differential equation
		\begin{align}\label{DDE_IIGS}
			\mathbb{D}^{\delta, \gamma}_{t}h_{\delta,\gamma}(x,t) &= -\frac{\partial}{\partial x}h_{\delta,\gamma}(x,t), x>0,
			\shortintertext{ with initial condition}
			h_{\delta,\gamma}(x,0) & =\delta(x),~~
			h_{\delta,\gamma}(0,t)  =\nu_G(t). \nonumber
		\end{align}
	\end{theorem}
	\begin{proof}
		Taking LT on both sides of \eqref{RL_IG}, it yields
		\begin{align*}
			\mathcal{L}_{t}\{\mathbb{D}^{\delta, \gamma}_t v(t)\}&= s \mathcal{L}_{t}(v(t))\left[\mathcal{L}_{t}\left(\sqrt{\frac{2}{\pi}}\delta (t)^{-1/2}e^{-\gamma^2t/2}\right)- \mathcal{L}_{t}\left(\frac{\delta \gamma}{\sqrt{\pi}}\Gamma{\left(1/2,\gamma^2t/2\right)}\right)\right]\nonumber\\
			& =  \delta s \mathcal{L}_{t}(v(t))\left[\frac{2}{\sqrt{(2s+\gamma^2)}}-\frac{\gamma}{s}\frac{\sqrt{(2s+\gamma^2)}-\gamma}{\sqrt{(2s+\gamma^2)}}\right]\nonumber\\
			& = \mathcal{L}_{t}(v(t)) \delta(\sqrt{2s+\gamma^2}-\gamma).
		\end{align*}
		\noindent Now,  applying  the LT with respect to $x$ on the both sides of \eqref{DDE_IIGS}, we get 
		$$
		\mathbb{D}^{\delta, \gamma}_{t}\mathcal{L}_{x}[h_{\delta,\gamma}(x,t)](y)= -y\mathcal{L}_{x}[h_{\delta,\gamma}(x,t)](y)-h_{\delta,\gamma}(0,t).
		$$
		Again, taking the LT with respect to $t$, we have that
		$$
		\delta(\sqrt{2s+\gamma^2}-\gamma)  \mathcal{L}_{t}[\mathcal{L}_{x}[h_{\delta,\gamma}(x,t)](y)](s)  = - y \mathcal{L}_{t}[\mathcal{L}_{x}[h_{\delta,\gamma}(x,t)](y)](s)+\frac{\delta(\sqrt{2s+\gamma^2}-\gamma)}{s}.
		$$
		We obtain
		$$
		\mathcal{L}_{t}[\mathcal{L}_{x}[h_{\delta,\gamma}(x,t)](y)](s)= \frac{\delta(\sqrt{2s+\gamma^2}-\gamma)}{s\left( y+\delta(\sqrt{2s+\gamma^2}-\gamma)\right)},
		$$
		which is the Laplace–Laplace transform of the \textit{pdf} $h_{\delta,\gamma}(x,t)$ (see \cite[Remark 2.1]{Kumar-Hitting}). This completes the proof.
	\end{proof}

		\begin{theorem}
			The \textit{pmf} $P_{\delta,\gamma}(n,t)=\mathbb{P}[Y_{\delta,\gamma}(t)=n]$ satisfy following fractional differential equation
			\begin{align*}
				\mathcal{D}^{\delta, \gamma}_{t} P_{\delta, \gamma}(n,t) &= -\lambda  P_{\delta, \gamma}(n,t) +\lambda \int_{-\infty}^{\infty} P_{\delta, \gamma}(n-x,t) F_X(x) dx.
			\end{align*}
		\end{theorem}
		\begin{proof}
			The proof is similar to the proof of Theorem \ref{Pro_DDE_GFCPP} and hence it is omitted here.
		\end{proof}
		
		\begin{spcase}
			Let $X_i,\;i=1,2,\ldots,$ are exponentially distributed with parameter $\eta$ in \eqref{def:IG-FCPP}. Then, the process $\{Y_f^\eta(t)\}_{t \geq 0}$ can be written as 
			\begin{align}\label{IIG_cpp_ex}
				Y_{\delta, \gamma}^\eta(t):= \sum_{i=1}^{N_{\delta, \gamma}(t)} X_i,~t\geq 0.
			\end{align}
			This process $\{Y_{\delta, \gamma}^\eta(t)\}_{t \geq 0}$ can also represented as $\{Y(E_{\delta, \gamma}(t))\}_{t \geq 0}$.
			\noindent The \textit{pdf} $P_{\delta, \gamma}^{\eta}(x,t)$ satisfies the following equation
			\begin{align*}
				\eta \mathcal{D}^{\delta, \gamma}_{t} P_{\delta, \gamma}^{\eta}(x,t)(x,t) &= -\left[\lambda +\mathcal{D}^{\delta, \gamma}(t)\right]\frac{\partial}{\partial x}  P_{\delta, \gamma}^{\eta}(x,t),
			\end{align*}
			where $ \mathcal{D}^{\delta, \gamma}_{t}$ fractional derivative \eqref{caputo_IG}, with initial condition
			$$
			P_{\delta, \gamma}^{\eta}(x,0)=0, \; \; 
			\mathbb{P}(Y_{\delta, \gamma}^\eta (t) >0)= 1-e^{-\lambda t}.
			$$
		\end{spcase}
		
		\begin{spcase}
			Let $X_i,i=1,2,\ldots$ be Mittag-Leffler distributed random variables with parameter $0<\beta<1$ and $ \eta>0$. Then, we define new process $\{Y_{\delta, \gamma}^{\beta, \eta}(t)\}_{ t\geq 0}$ such as
			$$
			Y_{\delta, \gamma}^{\beta, \eta}(t):= \sum_{i=1}^{N_{\delta,\gamma}(t)}X_{i},~t\geq 0.
			$$
			\noindent The LT of  $\{Y_{\delta, \gamma}^{\beta, \eta}(t)\}_{t \geq 0}$ is given by 
			\begin{align*}
				\mathbb{E}[e^{-sY_{\delta, \gamma}^{\beta, \eta}(t) }]= \mathbb{E}\left[e^{-\lambda E_{\delta, \gamma}(t)\frac{s^{\beta}}{\eta+s^{\beta}}}\right].
			\end{align*}
			The process $Y_{\delta, \gamma}^{\beta, \eta}(t)$ can be represented in terms of the $\beta$-stable subordinator, denoted by $D_{\beta}(t)$, time-changed with independent $\{Y_{\delta, \gamma}^{\eta}(t)\}_{t\geq0}$ \eqref{IIG_cpp_ex}, i.e. 
			$$
			Y_{\delta, \gamma}^{\beta, \eta}(t) \stackrel{d}{=} D_{\beta}(Y_{\delta, \gamma}^{\eta}(t)),~t\geq 0.
			$$
			The \textit{pdf} $P_{\delta, \gamma}^{\beta, \eta}(x,t)$ satisfies the following fractional differential equation
			\begin{align*}
				\eta \mathcal{D}^{\delta, \gamma}_{t} P_{\delta, \gamma}^{\beta, \eta}(x,t) &= -\left[\lambda +\mathcal{D}^{\delta, \gamma}_{t}\right]\mathcal{D}^{\beta}_{x} P_{\delta, \gamma}^{\beta, \eta}(x,t)\; 
			\end{align*}
			with initial conditions
			$$
			P_{\delta, \gamma}^{\beta, \eta}(x,0)=0, \; \; 
			\mathbb{P}(Y_{\delta, \gamma}^{\beta, \eta}(t) >0)= 1-\mathbb{E}[e^{-\lambda E_{\delta, \gamma}(t)}].
			$$
		\end{spcase}
		\begin{spcase} Let $X_i,i=1,2,\ldots$ are iid random variables with LT
			$$
			\mathbb{E}[e^{-sX_1}]=  \frac{\eta}{ \theta(\sqrt{2s+\chi^2}-\chi)+\eta}.
			$$
			Note that the distribution of $X_i$'s coincides with the distribution of inter-arrival times of inverse IG subordinated Poisson renewal process. The new stochastic process $\{Y_{\delta, \gamma}^{\theta, \chi}(t)\}_{t \geq 0}$ can be  defined as 
			\begin{align*}
				Y_{\delta, \gamma}^{\theta, \chi}(t): = \sum_{i=1}^{N_{\delta,\gamma}(t)}X_{i},~t\geq 0.
			\end{align*}
			The LT of \textit{pdf} of $\{Y_{\delta, \gamma}^{\theta, \chi}(t)\}_{t \geq 0}$ is given by 
			\begin{align*}
				\mathbb{E}[e^{-sY_{\delta, \gamma}^{\theta, \chi}(t) }]= \mathbb{E}\left[e^{-\lambda E_{\delta, \gamma}(t)\frac{\theta(\sqrt{2s+\chi^2}-\chi)}{\eta+\theta(\sqrt{2s+\chi^2}-\chi)}}\right].
			\end{align*}
			
			\noindent The process $\{Y_{\delta, \gamma}^{\theta, \chi}(t)\}_{t \geq 0}$ can be represented in terms of the IG subordinator, denoted by $\{D_{\theta, \chi}(t)\}_{t \geq 0}$, time-changed with independent $\{Y_{\delta, \gamma}^{\eta}(t)\}_{t\geq0}$ \eqref{IIG_cpp_ex}, i.e. 
			$$
			Y_{\delta, \gamma}^{\theta, \chi}(t) \stackrel{d}{=} D_{\theta, \chi}(Y_{\delta, \gamma}^{\eta}(t)),~t\geq 0.
			$$
			The \textit{pdf} $P_{\delta, \gamma}^{\theta, \eta, \chi}(x,t)$ satisfies the following fractional differential equation
			\begin{align*}
				\eta \mathcal{D}^{\delta, \gamma}_{t} P_{\delta, \gamma}^{\theta, \eta, \chi}(x,t) &= -\left[\lambda +\mathcal{D}^{\delta, \gamma}_{t}\right]\mathcal{D}^{\theta, \chi}_{x} P_{\delta, \gamma}^{\theta, \eta, \chi}(x,t)\; 
			\end{align*}
			with initial condition
			$$
			P_{\delta, \gamma}^{\theta, \eta, \chi}(x,0)=0, \; \; 
			\mathbb{P}(Y_{\delta, \gamma}^{\theta, \chi}(t) >0)= 1-\mathbb{E}[e^{-\lambda E_{\delta, \gamma}(t)}].
			$$
			
			\ifx 	\begin{spcase}
				Let $X_i, i=1,2,\cdot,$ are distributed as mentioned in \eqref{Dis_Bern}. Then the process $Y_{\delta, \gamma}^{g}$ is defined from the equation \eqref{Time_CFPP_Bern},
				\begin{equation}\label{Time_IIGFPP_Bern}
					Y_{\delta, \gamma}^{g}(t) = \sum_{i=1}^{N_{\delta, \gamma}(t)} X_i,
				\end{equation}
				is called inverse IG CPP with Bernst\'ein jumps.
			\end{spcase}
			\noindent These are the following results from Theorem \eqref{TC_Diff} in the subsection\eqref{GTFCPP_Bern}. This process has a time-changed representation, i.e.
			\begin{equation}
				Y_{\delta, \gamma}^{g}(t) \stackrel{d}{=}D_{g}(Y_{\delta, \gamma}^{\eta}(t)).
			\end{equation}
			The density $P_{\delta, \gamma}^{g}(x,t)$ satisfies the following equation
			\begin{align}
				\eta \mathcal{D}^{\delta, \gamma}_{t} P_{\delta, \gamma}^{g}(x,t) &= -\left[\lambda +\mathcal{D}^{\delta, \gamma}_{t}\right]\mathcal{D}^{g}_{x} P_{\delta, \gamma}^{g}(x,t)\; 
			\end{align}
			with conditions
			$$
			P_{\delta, \gamma}^{g}(x,0)=0, \; \; 
			\mathbb{P}(Y_{\delta, \gamma}^g (t) >0)= 1-\mathbb{E}[e^{-\lambda E_{\delta, \gamma}(t)}].
			$$\fi 
		\end{spcase}
\begin{spcase}
			Let $X_i,\;i=1,2,\ldots,$ be the iid discrete uniform random variables and $\{N_{\delta, \gamma}(t)\}_{ t \geq 0}$ be the time-changed Poisson process with inverse IG subordinator. Then, the process \eqref{CFPPOK} is  defined,
			\begin{equation*}
				Y_{\delta, \gamma}^{k}(t) := \sum_{i=1}^{N_{\delta, \gamma}(t)} X_i,,~t\geq 0
			\end{equation*}
			and called as the inverse IG fractional CPP of order $k$. The \textit{pmf} $P^{k}_{\delta, \gamma}(n,t)=\mathbb{P}[Y_{\delta, \gamma}^{k}(t)=n]$ satisfy following fractional differential-difference equation is given by 
			\begin{align*}
				\mathcal{D}^{\delta, \gamma}_{t} P^{k}_{\delta, \gamma}(n,t) &= -k \lambda \left(1-\frac{1}{k}\sum_{j=1}^{n \wedge k} B^{j}\right)P^{k}_{\delta, \gamma}(n,t),\; n>0,\\
				\mathcal{D}^{\delta, \gamma}_{t}P^{k}_{\delta, \gamma}(0,t)  &= - k \lambda P^{k}_{\delta, \gamma}(0,t).
			\end{align*} 
		\end{spcase}
		
		\begin{spcase}
			Let $X_i,i=1,2,\ldots,$ be iid truncated geometrically distributed random variables. From the equation \eqref{tsfcpp}, we define the process
			$$
			Y_{\delta, \gamma}^{\rho,k}(t): = \sum_{i=1}^{N_{\delta, \gamma}(t)} X_i,,~t\geq 0
			$$
			is called inverse IG fractional PAPoK.
			The following results can be obtained as a particular case of the Theorem \eqref{Diff_GFPAOK}. The \textit{pmf} $P^{\rho,k}_{\delta, \gamma}(n,t)=\mathbb{P}[Y_{\delta, \gamma}^{\rho,k}(t)=n]$ satisfy following fractional differential equation
			\begin{align*}
				\mathcal{D}^{\delta, \gamma}_{t} P^{\rho,k}_{\delta, \gamma}(n,t) &= -\lambda \left(1-\frac{1-\rho}{1-\rho^k}\sum_{j=1}^{n \wedge k}\rho^{j-1}B^j\right)  P^{\rho,k}_{\delta, \gamma}(n,t),\\
				\mathcal{D}^{\delta, \gamma}_{t} P^{\rho,k}_{\delta, \gamma}(0,t) &= -\lambda P^{\rho,k}_{\delta, \gamma}(0,t).
			\end{align*}
			\noindent We can express the the process $ \{Y_{\delta, \gamma}^{\rho,k}(t)\}_{t \geq 0}$ in a time-changed form as follows
			$$
			Y_{\delta, \gamma}^{\rho,k}(t)\stackrel{d}{=}N^k_A(E_{\delta, \gamma}(t)),~t\geq 0.
			$$
		\end{spcase}

		\noindent We now summarize the results obtained in this section in the following table. 
		
		\begin{table}[!ht]
			\centering
			
			\caption{Summary of results obtained in Section \ref{subsec:tmfcpp}}

			\begin{tabular}{SSSSSSSS} \toprule
				{Jump Size Distributions} & {DDE} & {Time-changed Representations}\ \ \\ \midrule
				{Exponential} & {$\eta \mathcal{D}^{\alpha, \mu}_{t}  = -\left[\lambda +\mathcal{D}^{\alpha, \mu}(t)\right]\frac{\partial}{\partial x} $} & {$Y(E_{\alpha, \mu}(t))$}\\
				{Mittag-Leffler (ML)}  & {$\eta \mathcal{D}^{\alpha, \mu}_{t}  = -\left[\lambda +\mathcal{D}^{\alpha, \mu}_{t}\right]\mathcal{D}^{\beta}_{x} $} & {$D_{\beta}(Y_{\alpha, \mu}^{\eta}(t))$}   \\
				{Tempered ML}  & {$\eta \mathcal{D}^{\alpha, \mu}_{t}  = -\left[\lambda +\mathcal{D}^{\alpha, \mu}_{t}\right]\mathcal{D}^{\beta, \nu}_{x} $} & {$D_{\beta, \nu}(Y_{\alpha, \mu}^{\eta}(t))$}  \\
				{Inter-times Bernst\'ein}  & {$\eta \mathcal{D}^{\alpha,\mu}_{t}  = -\left[\lambda +\mathcal{D}^{\alpha,\mu}_{t}\right]\mathcal{D}^{g}_{x}$} & {$D_{g}(Y_{\alpha,\mu}^{\eta}(t))$}    \\ 
				{Truncated geometric}  & {$\mathcal{D}^{\alpha, \mu}_{t} = -\lambda \left(1-\frac{1-\rho}{1-\rho^k}\sum_{j=1}^{n \wedge k}\rho^{j-1}B^j\right)$} & {$N_A^k(E_{\alpha,\mu}(t))$}   \\
				{Discrete uniform}  & {$\mathcal{D}^{\alpha,\mu}_{t} =- \left[{\left( \mu + k \lambda \left(1-\frac{1}{k}\sum_{j=1}^{n \wedge k} B^{j}\right)\right)}^{\alpha} - \mu^{\alpha}\right]$} & {$N^k(E_{\alpha,\mu}(t))$}  \\ \bottomrule
			\end{tabular}
		\end{table}
		\begin{table}[!ht]
			\centering
			
			\caption{Summary of results obtained in Section \ref{subsec:iig}}
			
			\begin{tabular}{SSSS} \toprule
				{Jump Size Distributions} & {DDE} & {Time-changed Representations}\ \ \\ \midrule
				{Exponential} & {$\eta \mathcal{D}^{\delta, \gamma}_{t}  = -\left[\lambda +\mathcal{D}^{\delta, \gamma}(t)\right]\frac{\partial}{\partial x} $} & {$Y(N_{\delta, \gamma}(t))$}\\
				{Mittag-Leffler (ML)}  & {$\eta \mathcal{D}^{\delta, \gamma}_{t}  = -\left[\lambda +\mathcal{D}^{\delta, \gamma}_{t}\right]\mathcal{D}^{\beta}_{x} $} & {$D_{\beta}(Y_{\delta, \gamma}^{\eta}(t))$}   \\
				{Inter-times inverse of IGS}  & {$\eta \mathcal{D}^{\delta, \gamma}_{t}  = -\left[\lambda +\mathcal{D}^{\delta, \gamma}_{t}\right]\mathcal{D}^{\theta, \chi}_{x} $} & {$D_{\theta, \chi}(Y_{\delta, \gamma}^{\eta}(t))$}  \\
				{Inter-times Bernst\'ein}  & {$\eta \mathcal{D}^{\delta, \gamma}_{t}  = -\left[\lambda +\mathcal{D}^{\delta, \gamma}_{t}\right]\mathcal{D}^{g}_{x}$} & {$D_{g}(Y_{\delta, \gamma}^{\eta}(t))$}    \\ 
				{Truncated geometric}  & {$\mathcal{D}^{\delta, \gamma}_{t} = -\lambda \left(1-\frac{1-\rho}{1-\rho^k}\sum_{j=1}^{n \wedge k}\rho^{j-1}B^j\right)$} & {$N_A^k(E_{\delta, \gamma}(t))$}   \\
				{Discrete uniform}  & {$\mathcal{D}^{\delta, \gamma}_{t} =- {\left[] k \lambda \left(1-\frac{1}{k}\sum_{j=1}^{n \wedge k} B^{j}\right)\right]}^{\alpha}$} & {$N^k(E_{\delta, \gamma}(t))$}  \\ \bottomrule
			\end{tabular}
		\end{table}
		
		\section{Simulations}\label{sec:sim}
		\noindent In this section, we simulate the sample trajectories of the special cases of the  GFCPP $\{Y_f(t)\}_{t\geq 0}$ \eqref{GFCPP}. 
		First, we reproduce an algorithm for generating the sample paths for CPP $\{Y(t)\}_{t\geq 0}$ \eqref{CPP} with the jump size distribution $F_X$. 
		\begin{algorithm}[H]
			\caption{Simulation of the CPP}\label{algo-ing}
			\begin{algorithmic}[1]
				\renewcommand{\algorithmicrequire}{\textbf{Input:}}
				\renewcommand{\algorithmicensure}{\textbf{Output:}}
				\REQUIRE $\lambda>0 $, $T\geq0$, and $\theta$ (parameter of $F_X$).
				\ENSURE  $Y(t)$,  simulated sample paths for the CPP with jump size distribution $F_X$.
				\\ \textit{Initialisation} : $t=0$, $Y=0$ and $v=0$.
				\WHILE{$t < T$} 
				\STATE generate a uniform random variable $U \sim U(0, 1)$.
				\STATE set $t \leftarrow t +\left[-\frac{1}{\lambda}\log{U}\right]$.
				\STATE generate an independent random variable $X$ with distribution $F_X$ with parameter $\theta$.
				\STATE set $v \leftarrow v + X$ and append in $Y$.
				\ENDWHILE
				\STATE 	\textbf{return }$Y.$
			\end{algorithmic}
		\end{algorithm}
		Here, $Y$ denotes the number of events that occurred up to time $T>0$.\\
		Now, we present algorithm for time-changed by IIGS and ITSS  processes as we use the corresponding algorithms mentioned above for respective subordinators. 
		
		\begin{algorithm}[H]
			
			\caption{Simulation for time-changed CPP with IIGS (ITSS)}\label{sim:tcpp}
			\begin{algorithmic}[1]
				\renewcommand{\algorithmicrequire}{\textbf{Input:}}
				\renewcommand{\algorithmicensure}{\textbf{Output:}}
				\REQUIRE Parameter $\delta$ and $\gamma$ for respective IIGS $E_{\delta,\gamma}(t)$ (ITSS $E_{\alpha, \mu}(t)$), $\lambda>0$, $T\geq0$, $\eta>0$.
				\ENSURE The sample paths of the subordinated Poisson process $Y(E_{\delta,\gamma}(t_i)),\ 1 \leq i \leq n.$
				\\ \textit{Initialisation} :
				\STATE set $dt = \frac{T}{n}$ and choose $n+1$ uniformly spaced time points $0=t_0, t_1, \dots, t_n=T$ with $dt=t_1-t_0$.
				
				\STATE  generate the values $E_{\delta,\gamma}(t_i),\ 1 \leq i \leq n,$ ( $E_{\alpha,\mu}(t)$) for the IIGS (ITSS) at time points $t_1, t_2, ..., t_n$ using the Algorithms from \cite{TCFPP-pub}
				
				\STATE use the values $E_{\delta,\gamma}(t_i),\ 1 \leq i \leq n,$ generated in Step 2, as time points and compute the number of events of the subordinated CPP $Y(E_{\delta,\gamma}(t_i)),\ 1 \leq i \leq n,$ using Algorithm \ref{algo-ing}.
				\STATE \textbf{return}	  $Y(E_{\delta,\gamma}(t_i))$
			\end{algorithmic}
		\end{algorithm}
		Using Algorithm \ref{algo-ing} and \ref{sim:tcpp}, we generate the sample paths for the chosen set of parameters which is given below 
		\begin{table}[!ht]
			\centering
			\begin{tabular}{|c|p{2.8cm}|p{2.8cm}|p{2.8cm}|p{2.8cm}|}
				\hline &Exponential& Mittag-Leffler & Discrete Uniform& Discrete logarithm   \\\hline 
				CPP (Fig \ref{fig:cpp}) & $\eta=2,\lambda=4$ & $\eta=2, \lambda=4, \alpha=0.9$ & $\lambda=4, k=5$ & $\lambda=4, q=0.5$\\
				CPP-IIGN (Fig \ref{fig:iign})&$\eta=2,\lambda=4, \delta=0.3, \gamma=1$ &$\eta=2,\lambda=4, \delta=0.3, \gamma=1, \beta=0.9$ &$\lambda=4, \delta=0.3, \gamma=1, k=5$& $\lambda=4, \delta=0.3, \gamma=1, q=0.5$\\
				CPP-ITSS (Fig \ref{fig:itss})&$\eta=2,\lambda=4, \alpha=0.7, \mu=2$ &$\eta=2,\lambda=4, \alpha=0.7, \mu=2, \beta= 0.9$&$\lambda=4, \alpha=0.7, \mu=2, k=5$ &$\lambda=4, \alpha=0.7, \mu=2, q= 0.5$\\\hline
				
			\end{tabular}
			\caption{Table of parameters for the generated sample paths}
			
		\end{table}
		\begin{figure}[!ht]
			\centering
			\begin{minipage}{.2\textwidth}
				\centering
				\includegraphics[width=1\linewidth]{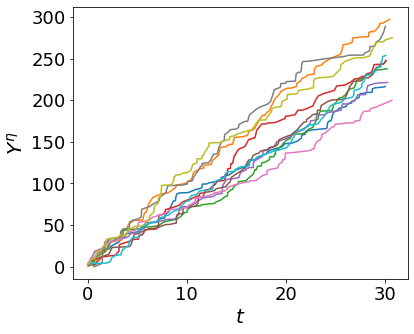}
			\end{minipage}%
			\begin{minipage}{.2\textwidth}
				\centering
				\includegraphics[width=1\linewidth]{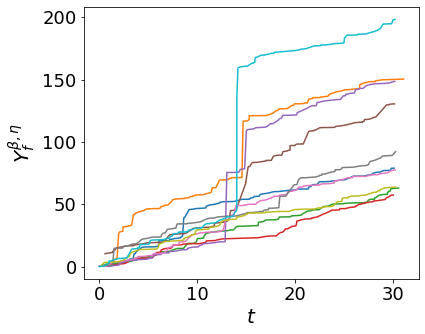}
			\end{minipage}
			\begin{minipage}{.2\textwidth}
				\centering
				\includegraphics[width=1\linewidth]{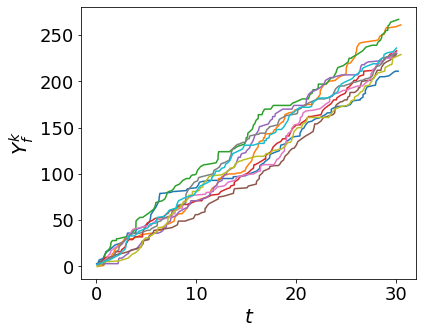}
			\end{minipage}
			\begin{minipage}{.2\textwidth}
				\centering
				\includegraphics[width=1\linewidth]{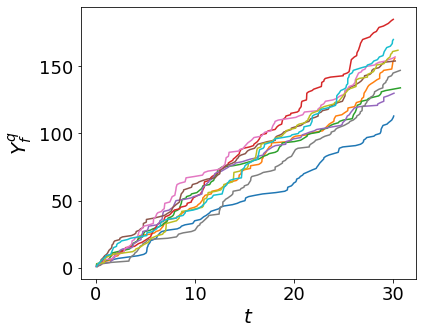}
			\end{minipage}
			
			\caption{\small{CPP with exponential, Mittag-Leffler, discrete uniform, discrete logarithmic jumps sample trajectories }}\label{fig:cpp}
		\end{figure}
		\begin{figure}[!ht]
			\centering
			\begin{minipage}{.2\textwidth}
				\centering
				\includegraphics[width=1\linewidth]{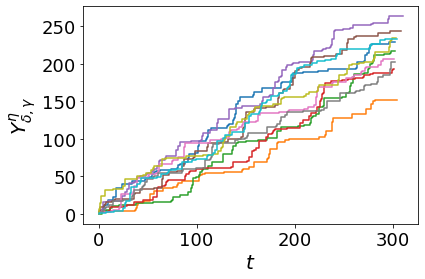}
			\end{minipage}%
			\begin{minipage}{.2\textwidth}
				\centering
				\includegraphics[width=1\linewidth]
				{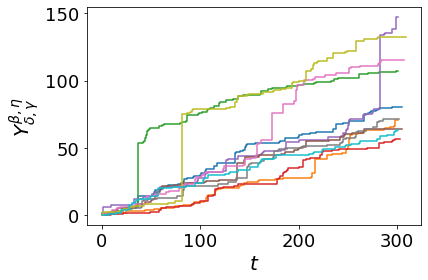}

			\end{minipage}
			\begin{minipage}{.2\textwidth}
				\centering
				\includegraphics[width=1\linewidth]
				{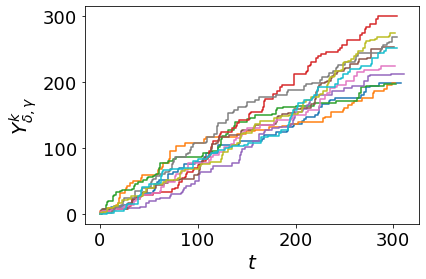}

			\end{minipage}
			\begin{minipage}{.2\textwidth}
				\centering
				\includegraphics[width=1\linewidth]
				{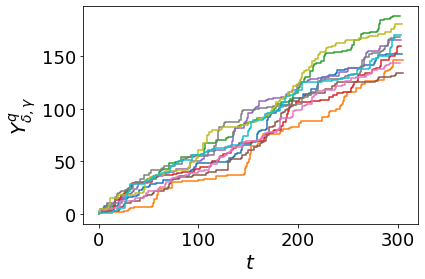}
			\end{minipage}
			\caption{\small{Inverse IG fractional CPP with exponential, Mittag-Leffler, discrete uniform, discrete logarithmic jumps sample trajectories }}\label{fig:iign}
		\end{figure}
		\begin{figure}[!ht]
			\centering
			\begin{minipage}{.2\textwidth}
				\centering
				\includegraphics[width=1\linewidth]{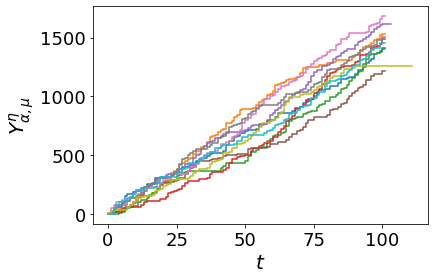}
			\end{minipage}%
			\begin{minipage}{.2\textwidth}
				\centering
				\includegraphics[width=1\linewidth]{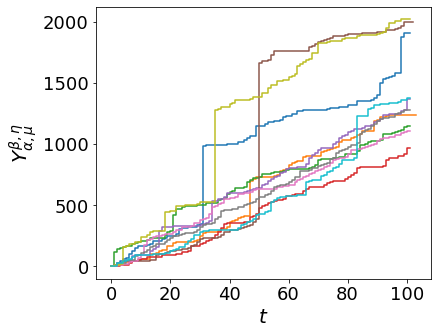}
			\end{minipage}
			\begin{minipage}{.2\textwidth}
				\centering
				\includegraphics[width=1\linewidth]{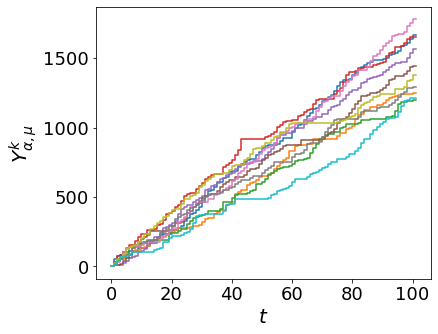}
			\end{minipage}
			\begin{minipage}{.2\textwidth}
				\centering
				\includegraphics[width=.9\linewidth]{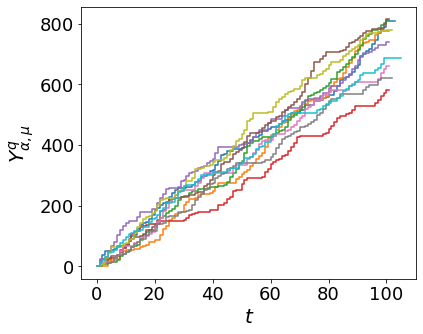}
			\end{minipage}
			
			\caption{ \small{Tempered fractional CPP exponential, Mittag-Leffler, discrete uniform, discrete logarithmic jumps sample trajectories}}\label{fig:itss}
		\end{figure}

		\nomenclature{$Y_f(t)$}{Generalized fractional compound Poisson process (GFCPP)}
		\nomenclature{$Y_f^{\eta}(t)$}{GFCPP with exponential jumps}
		\nomenclature{$Y_f^{\beta, \eta}(t)$}{GFCPP with Mittag-Leffler jumps}
		\nomenclature{$Y_f^g(t)$}{GFCPP with Bernst\'ein jumps}
		\nomenclature{$Y_f^k(t)$}{Generalized fractional Poisson process of order $k$ (GFPPoK)}
		\nomenclature{$Y_f^{\rho, k}(t)$}{Generalized fractional P\'olya Aeppli process of order $k$ (GPAPoK)}
		\nomenclature{$Y_f^q(t)$}{Generalized fractional negative binomial process (GFNBP)}
		\nomenclature{$Y_{\alpha, \mu}(t)$}{Tempered fractional CPP (TFCPP)}
		\nomenclature{$Y_{\alpha, \mu}^{\eta}(t)$}{TFCPP with exponential jump}
		\nomenclature{$Y_{\alpha, \mu}^{\beta,\eta, \nu}(t)$}{TFCPP with tempered Mittag-Leffler jumps}
		\nomenclature{$Y_{\alpha, \mu}^{g}(t)$}{TFCPP with Bernst\'ein jumps}
		\nomenclature{$Y_{\alpha, \mu}^{\rho,k}(t)$}{Tempered fractional PAPoK}
		\nomenclature{$Y_{\alpha, \mu}^{k}(t)$}{Tempered fractional PPoK}
		\nomenclature{$Y_{\delta, \gamma}(t)$}{Inverse IG fractional CPP}
		\nomenclature{$Y_{\delta, \gamma}^k(t)$}{Inverse IG fractional PPoK}
		\nomenclature{$Y_{\delta, \gamma}^{\rho,k}(t)$}{Inverse IG fractional PAPoK}
		\nomenclature{$Y_{\delta, \gamma}^{\eta}(t)$}{Inverse IG fractional CPP with exponential jumps}
		\nomenclature{$Y_{\delta, \gamma}^{\beta,\eta}(t)$}{Inverse IG fractional CPP with Mittag-Leffler jumps}
		\nomenclature{$Y_{\delta, \gamma}^{\theta, \chi}(t)$}{Inverse IG fractional CPP with IG type jumps}
		\nomenclature{$Y_{\delta, \gamma}^{g}(t)$}{Inverse IG fractional CPP with Bernst\'ein jumps}


	\def\cprime{$'$}
	
			\printnomenclature
\end{document}